\documentclass{amsart}
\usepackage{amsmath,amssymb}
\usepackage{amssymb}
\usepackage{mathtools}
\makeatletter
\newcommand{\vast}{\bBigg@{3}}
\newcommand{\Vast}{\bBigg@{4}}
\newcommand{\vastl}{\mathopen\vast}
\newcommand{\vastr}{\mathclose\vast}
\makeatother

\usepackage[english]{babel}
\usepackage[margin=35mm]{geometry}

\newtheorem{theorem}{Theorem}[section]
\newtheorem{proposition}[theorem]{Proposition}
\newtheorem{corollary}[theorem]{Corollary}
\newtheorem{lemma}[theorem]{Lemma}

\newtheorem*{theorem*}{Theorem}
\newtheorem*{corollary*}{Corollary}

\theoremstyle{definition}

\numberwithin{equation}{section}

\newcommand{\dN}{\mathbb N}                     
\newcommand{\dZ}{\mathbb Z}                     
\newcommand{\dR}{\mathbb R}                     
\newcommand{\dC}{\mathbb C}                     
\newcommand{\ra}{\rightarrow}                   
\DeclareMathOperator{\e}{\mathrm{e}}                
\allowdisplaybreaks     
\title[Ostrowski sum-of-digits function]%
{Pseudorandomness of the Ostrowski sum-of-digits function}
\author{Lukas Spiegelhofer}
\address{JKU Linz, Austria}

\thanks{
The author acknowledges support by the Austrian Science Fund (FWF), project F5505-N26, which is a part of the Special Research Program ``Quasi Monte Carlo Methods: Theory and Applications''.
}

\begin{document}
\maketitle
\begin{abstract}
For an irrational $\alpha\in(0,1)$, we investigate the Ostrowski sum-of-digits function $\sigma_\alpha$.
For $\alpha$ having bounded partial quotients and $\vartheta\in\mathbb R\setminus\mathbb Z$,
we prove that the function $g:n\mapsto \mathrm e(\vartheta \sigma_\alpha(n))$, where $\mathrm e(x)=\mathrm e^{2\pi i x}$, is pseudorandom in the following sense:
for all $r\in\mathbb N$ the limit
\[\gamma_r=
\lim_{N\rightarrow\infty}\frac 1N\sum_{0\leq n<N}g(n+r)\overline{g(n)} \]
exists and we have
\[\lim_{R\rightarrow\infty}\frac 1R\sum_{0\leq r<R}\bigl\lvert \gamma_r\bigr\rvert^2=0.\]

\end{abstract}
\section{Introduction and main results}
Let $g$ be an arithmetical function.
The set of      
$\beta\in[0,1)$ 
satisfying
\[
\limsup_{N\rightarrow\infty}\frac 1N
\biggl\lvert
\sum_{n<N}
g(n)\e(-n\beta)
\biggr\rvert
>0
\]
is called the \emph{Fourier--Bohr spectrum} of $g$.

The function $g$ is called \emph{pseudorandom in the sense of Bertrandias}~\cite{B1964}
or simply \emph{pseudorandom} if the limit
\[
\gamma_r
=
\lim_{N\rightarrow\infty}
\frac 1N
\sum_{0\leq n<N}g(n+r)\overline{g(n)}
\]
exists for all $r\geq 0$ and the family $\gamma$ is zero in quadratic mean, that is,
\begin{equation*}
\lim_{R\rightarrow\infty}
\frac 1R
\sum_{0\leq r<R}\bigl\lvert \gamma_r\bigr\rvert^2
=
0
.
\end{equation*}
(We note that by the Cauchy--Schwarz inequality this is equivalent to
$\frac 1R\sum_{r<R}\bigl\lvert \gamma_r\bigr\rvert=o(1)$
for bounded $g$.)
Pseudorandomness can be understood as a property of the \emph{spectral measure} associated to $g$:
Assume that the autocorrelation $\gamma$ of $g$ exists.
By the Bochner representation theorem there exists a unique measure $\mu$ on the Torus $\mathbb T=\mathbb R/\mathbb Z$ such that
\[\gamma_r=\int_{\mathbb T}\e(rx)\mathrm d\mu(x)
\]
for all $r$. Then $g$ is pseudorandom if and only if the discrete component of $\mu$ vanishes.
We refer to~\cite{CKM1977} for more details.

It is known that pseudorandomness of a bounded arithmetic function $g$ implies that the spectrum of $g$ is empty, which can be proved using van der Corput's inequality.
For the convenience of the reader, we give a proof of this fact in Section~\ref{sec:lemmas}.

The converse of this statement does not always hold.
However, it is true for $q$-multiplicative functions $g:\dN\rightarrow \mathbb T=\{z\in\dC:\lvert z\rvert =1 \}$, which has been proved by Coquet
\cite{C1976,C1978,C1980}.
Here a function $g:\dN\rightarrow \dC$ is called $q$-\emph{multiplicative} if
$f\bigl(q^kn+b\bigr)=f\bigl(q^kn\bigr)f(b)$ for all integers $k,n>0$ and $0\leq b<q^k$.

The purpose of this paper is to prove an analogous statement for the Ostrowski numeration system, that is, for $\alpha$-\emph{multiplicative} functions.
Assume that $\alpha\in(0,1)$ is irrational.
The Ostrowski numeration system has as its scale of numeration the sequence of denominators of the convergents of the regular continued fraction expansion of $\alpha$.
More precisely, let $\alpha=[0;a_1,a_2,\ldots]$ be the continued fraction expansion of $\alpha$ and $p_i/q_i=[0;a_1,\ldots,a_i]$ the $i$-th convergent to $\alpha$, where $i\geq 0$.
By the greedy algorithm, every nonnegative integer $n$ has a representation
\begin{equation}\label{eqn:greedy}
n=\sum_{k\geq 0}\varepsilon_k q_k
\end{equation}
such that
\[\sum_{0\leq k<K}\varepsilon_kq_k<q_K\] for all $K\geq 0$.
This algorithm yields the unique expansion of the form~\eqref{eqn:greedy} having the properties that
$0\leq \varepsilon_0<a_1$, $0\leq \varepsilon_k\leq a_{k+1}$ and
$\varepsilon_k=a_{k+1}\Rightarrow \varepsilon_{k-1}=0$ for $k\geq 1$,
the \emph{Ostrowski expansion of} $n$.

For a nonnegative integer $n$ let $(\varepsilon_k(n))_{k\geq 0}$ be its Ostrowski expansion.
An arithmetic function $g$ is $\alpha$-\emph{additive} resp. $\alpha$-\emph{multiplicative}
if
\[f(n)=\sum_{k\geq 0}f\bigl(\varepsilon_k(n)q_k\bigr)\quad\textrm{resp.}\quad
f(n)=\prod_{k\geq 0}f\bigl(\varepsilon_k(n)q_k\bigr)\]
for all $n$.
Examples of $\alpha$-additive functions are the functions $n\mapsto \beta n$ and the $\alpha$-\emph{sum of digits} of $n$~\cite{C1980b}:
\[\sigma_\alpha(n)=\sum_{i\geq 0}\varepsilon_i(n).\]
We refer the reader to~\cite{B2001} for a survey on the Ostrowski numeration system.
In particular, we want to note that the Ostrowski numeration system is a useful tool for studying the discrepancy modulo $1$ of $n\alpha$-sequences,
see for example the references contained in the aforementioned paper.

Moreover, see~\cite{BL2004} for a dynamical viewpoint of the Ostrowski numeration system, and~\cite{GLT1995, BBLT2006} for more general numeration systems.

Our main theorem establishes a connection between the Fourier--Bohr spectrum and pseudorandomness for $\alpha$-multiplicative functions.
\begin{theorem}\label{thm:main}
Assume that $g$ is a bounded $\alpha$-multiplicative function.
The Fourier--Bohr spectrum of $g$ is empty if and only if $g$ is pseudorandom.
\end{theorem}
Using a theorem by Coquet, Rhin and Toffin~\cite[Theorem~2]{CRT1983},
we obtain the following corollary.
\begin{corollary}
Assume that $\alpha\in(0,1)$ is irrational and has bounded partial quotients and $\vartheta\in\dR\setminus\dZ$.
Then $n\mapsto \e(\vartheta \sigma_\alpha(n))$ is pseudorandom.
\end{corollary}
In particular, this holds for the \emph{Zeckendorf sum-of-digits function}, which corresponds to the case $\alpha=\bigl(\sqrt{5}-1\bigr)/2=[0;1,1,\ldots]$.
This special case can be found in the author's thesis~\cite{S2014}.

We first present a series of auxiliary results, and proceed to the proof of Theorem~\ref{thm:main} in section~\ref{sec:thmproof}.
\section{Lemmas}\label{sec:lemmas}
We begin with the well-known inequality of van der Corput.
\begin{lemma}[Van der Corput's inequality]\label{lem:vdc}
Let $I$ be a finite interval in $\dZ$ and
let $a_n\in\dC$ for $n\in I$.
Then
\[
\Biggl\lvert\sum_{n\in I}a_n\Biggr\rvert^2\leq
\frac{\lvert I\rvert-1+R}R\sum_{0\leq\lvert r\rvert<R}\biggl(1-\frac{\lvert r\rvert}R\biggr)
\sum_{\substack{n\in I\\n+r\in I}}a_{n+r}\overline{a_n}
\]
for all integers $R\geq 1$.
\end{lemma}
In the definition of pseudorandomness for bounded arithmetic functions $g$,
we do not actually need the square.
\begin{lemma}\label{lem:L1}
Let $g$ be a bounded arithmetic function such that the correlation of $g$ exists.
The function $g$ is pseudorandom if and only if
\[\lim_{R\rightarrow\infty}
\frac 1R
\sum_{0\leq r<R}\lvert \gamma_r\rvert =0.
\]
\end{lemma}
For the proof of sufficiency we note that we may without loss of generality assume that $\lvert g\rvert\leq 1$. The other direction is an application of the Cauchy-Schwarz inequality.

As we noted before, pseudorandomness of $g$ implies that the spectrum of $g$ is empty.
\begin{lemma}\label{lem:pseudorandom_spectrum}
Let $g$ be a bounded arithmetic function.
If $g$ is pseudorandom, then the Fourier--Bohr spectrum of $g$ is empty.
\end{lemma}
\begin{proof}
The proof is an application of van der Corput's inequality (Lemma \ref{lem:vdc}).
We have for all $R\in \{1,\ldots,N\}$
\begin{align*}
\Biggl\lvert\frac 1N\sum_{0\leq n<N}g(n)\e(n\beta)\Biggr\rvert^2
&\leq
\frac{N-1+R}{RN^2}\sum_{0\leq \lvert r\rvert<R}\biggl(1-\frac{\lvert r\rvert}{R}\biggr)
\e(r\beta)
\sum_{0\leq n,n+r<N}g(n+r)\overline{g(n)}
\\
&\ll
\frac 1R
\sum_{0\leq r<R}
\Biggl\lvert
\frac 1N
\sum_{0\leq n<N}
g(n+r)\overline{g(n)}
\Biggr\rvert
+O\biggl(\frac RN\biggr)
.
\end{align*}
Let $\varepsilon\in (0,1)$.
By hypothesis and Lemma~\ref{lem:L1} we may choose $R$ so large that 
\[
\frac 1R
\sum_{0\leq r<R}
\bigl\lvert\gamma_r\bigr\rvert
<
\varepsilon^2
.
\]
Moreover, we choose $N_0$ in such a way that
$R/N_0<\varepsilon^2$ and
\[
\Biggl\lvert
\frac 1N
\sum_{0\leq n<N}
g(n+r)\overline{g(n)}
-\gamma_r
\biggr\rvert <\varepsilon^2
\]
for all $r<R$ and $N\geq N_0$.
Then for $N\geq N_0$ we have
\[
\Biggl\lvert\frac 1N\sum_{0\leq n<N}g(n)\e(n\beta)\biggr\rvert^2
\ll
\frac 1R
\sum_{0\leq r<R}
\lvert\gamma_r\rvert
+
\frac 1R
\sum_{0\leq r<R}
\Biggl\lvert
\frac 1N
\sum_{0\leq n<N}
g(n+r)\overline{g(n)}
-\gamma_r
\Biggr\rvert
+
O\biggl(\frac R{N_0}\biggr)
<3\varepsilon^2
.
\]
\end{proof}

The following lemma is a generalization of Dini's Theorem.
\begin{lemma}\label{lem:dini}
Assume that $(f_i)_{i\geq 0}$ is a sequence of nonnegative continuous functions on $[0,1]$ converging pointwise to the zero function.
Assume that $\lvert f_{i+1}(x)\rvert\leq \max\bigl\{\lvert f_i(x)\rvert,\lvert f_{i-1}(x)\rvert\bigr\}$.
Then the convergence is uniform in $x$.
\end{lemma}
\begin{proof}
For $\varepsilon>0$, for nonnegative $N$ and $x\in[0,1]$ we set
\[A_N(x)=\{\xi\in[0,1]:f_N(\xi)<\varepsilon\mbox{ and }f_{N+1}(\xi)<\varepsilon\}.\]
Note that this is an open set.
By induction, using the property $\lvert f_{i+1}(x)\rvert\leq \max\bigl\{\lvert f_i(x)\rvert,\lvert f_{i-1}(x)\rvert\bigr\}$, we obtain
\[A_N(x)=\{\xi\in[0,1]:f_n(\xi)<\varepsilon\mbox{ for all }n\geq N\}.\]
Trivially, we have $A_N(x)\subseteq A_{N+1}(x)$.
For each $x\in [0,1]$ there is an $N(x)$ such that $f_n(x)<\varepsilon$ for all $n\geq N(x)$.
Then $x\in A_{N(x)}(x)$, therefore $\bigl(A_{N(x)}(x)\bigr)_{x\in[0,1]}$ is an open cover of the compact set $[0,1]$.
Choose $x_1,\ldots,x_k$ and $N_1,\ldots,N_k$ such that
$A_{N_1}(x_1)\cup\cdots\cup A_{N_k}(x_k)=[0,1]$ and set $N=\max\{N_1,\ldots,N_k\}$.
By monotonicity of the sets $A_N(x)$, we obtain
$A_N(x_1)\cup\cdots\cup A_N(x_k)=[0,1]$, in other words,
$f_n(\xi)<\varepsilon$ for all $\xi\in[0,1]$ and all $n\geq N$.
\end{proof}
\begin{lemma}\label{lem:enumeration}
Let $(w_i)_i$ be the increasing enumeration of the integers
$n$ such that $\varepsilon_0(n)=\cdots=\varepsilon_{\lambda-1}(n)=0$.
The intervals    
$[w_i,w_{i+1})$  
constitute a partition of the set $\dN$ into intervals of length $q_\lambda$ and $q_{\lambda-1}$,
where $w_{i+1}-w_i=q_{\lambda-1}$ if and only if $\varepsilon_\lambda\bigl(w_i\bigr)=a_{\lambda+1}$.
\end{lemma}
\begin{proof}
Assume first that $\varepsilon_\lambda(w_i)=a_{\lambda+1}$.
We want to show that $w_{i+1}=w_i+q_{\lambda-1}$.
Let $w_i\leq n<w_i+q_{\lambda-1}$.
Then the Ostrowski expansion of $n$ is obtained by superposition of the expansions of $w_i$ and of $n-w_i$.
In particular, for $w_i<n<w_i+q_{\lambda-1}$ we have $\varepsilon_j(n)\neq 0$ for some $j<\lambda-1$.
Moreover, in the addition $w_i+q_{\lambda-1}$ a carry occurs, producing $\varepsilon_j(w_i+q_{\lambda-1})=0$ for $j\leq\lambda$, therefore $w_{i+1}=w_i+q_{\lambda-1}$.
The case $\varepsilon_\lambda(w_i)<a_{\lambda+1}$ is similar, in which case
$w_{i+1}=w_i+q_\lambda$.
\end{proof}
For an $\alpha$-multiplicative function $g$ and an integer $\lambda\geq 0$ we define a function $g_\lambda$ by truncating the digital expansion:
we define
$\psi_\lambda(n)=\sum_{i<\lambda}\varepsilon_i(n)q_i$ and
\[g_\lambda(n)=g\bigl(\psi_\lambda(n)\bigr).\]
We will need the following carry propagation lemma for the Ostrowski numeration system.
\begin{lemma}\label{lem:alpha_carry_lemma}
Let $\lambda\geq 1$ be an integer and $N,r\geq 0$.
Assume that $\alpha\in(0,1)$ is irrational and let $g$ be an $\alpha$-multiplicative function.
Then
\begin{equation}\label{eqn:carry}
\Bigl\lvert
\bigl\{
n<N:g(n+r)\overline {g(n)}\neq g_\lambda(n+r)\overline {g_\lambda(n)}
\bigr\}
\Bigr\rvert\leq
N\frac r{q_{\lambda-1}}
.
\end{equation}

\end{lemma}
\begin{proof}


The statement we want to prove is trivial for $r\geq q_{\lambda-1}$, we assume therefore that $r<q_{\lambda-1}$.
Let $w$ be the family from Lemma~\ref{lem:enumeration}.
For $w_i\leq n<w_{i+1}-r$,
we have $\varepsilon_j(n+r)=\varepsilon_j(n)$ for $j\geq \lambda$.
It follows that
\[
\Bigl\lvert\bigl\{n\in \{w_i,\ldots,w_{i+1}-1\}:g(n+r)\overline{g(n)}\neq g_\lambda(n+r)\overline{g_\lambda(n)}\bigr\}\Bigr\rvert\leq r
.
\]
By concatenating blocks, the statement follows therefore for the case that $N=w_i$ for some $i$.
It remains to treat the case that $w_i<N<w_{i+1}$ for some $i$.
To this end, we denote by $L(N)$ resp. $R(N)$ the left hand side resp. the right hand side of~\eqref{eqn:carry}.
For $w_i\leq N\leq w_{i+1}$ we have
\[L(N)=\begin{cases}L(w_i),&N\leq w_{i+1}-r;\\L(w_i)+N-(w_{i+1}-r),&N\geq w_{i+1}-r.\end{cases}\]
Note that $L(N)$ is a polygonal line that lies below $R(N)$ for
$N\in\{w_i,w_{i+1}-r,w_{i+1}\}$ and therefore for all $N\in[w_i,w_{i+1}]$.
By concatenating blocks, the full statement follows.
\end{proof}

We define Fourier coefficients for $g$:
\[
G_\lambda(h)=\frac 1{q_\lambda}
\sum_{0\leq u<q_\lambda}g(u)\e\bigl(huq_\lambda^{-1}\bigr)
.
\]
\begin{lemma}\label{lem:correlation}
Assume that $i$ be such that $w_{i+1}-w_i=q_\lambda$ and let $r\geq 0$. We have
\begin{equation}\label{eqn:correlation_partial_sums}
\sum_{h<q_\lambda}\bigl\lvert G_\lambda(h)\bigr\rvert^2\e\bigl(hrq_\lambda^{-1}\bigr)
=
\frac 1{q_\lambda}
\sum_{w_i\leq u<w_{i+1}}
g_\lambda(v+r)\overline{g_\lambda(v)}
+
O\biggl(\frac r{q_\lambda}\biggr)
.
\end{equation}
\end{lemma}
\begin{proof}
\begin{align*}
\sum_{0\leq h<q_\lambda}\bigl\lvert G_\lambda(h)\bigr\rvert^2\bigl(hrq_\lambda^{-1}\bigr)
&=
\frac 1{q_\lambda}\sum_{0\leq u,v<q_\lambda}g_\lambda(u)\overline{g_\lambda}(v)
\frac 1{q_\lambda}
\sum_{0\leq h<q_\lambda}\e\biggl(\frac h{q_\lambda}(v+r-u)\biggr)
\\
&=
\frac 1{q_\lambda}
\sum_{0\leq u,v<q_\lambda}\left[\!\left[v+r\equiv u\bmod q_\lambda\right]\!\right]g_\lambda(u)\overline{g_\lambda(v)}
\\
&=
\frac 1{q_\lambda}
\sum_{w_i\leq u,v<w_{i+1}}\left[\!\left[v+r\equiv u\bmod q_\lambda\right]\!\right]
g_\lambda(u)\overline{g_\lambda(v)}
\\
&=
\frac 1{q_\lambda}
\sum_{w_i\leq u<w_{i+1}-r}
g_\lambda(v+r)\overline{g_\lambda(v)}
+
O\biggl(\frac r{q_\lambda}\biggr)
\\
&=
\frac 1{q_\lambda}
\sum_{w_i\leq u<w_{i+1}}
g_\lambda(v+r)\overline{g_\lambda(v)}
+
O\biggl(\frac r{q_\lambda}\biggr)
.
\end{align*}
\end{proof}
\begin{lemma}\label{lem:sieve}
Let $H\geq 1$ be an integer and $R$ a real number.
For all real numbers $t$ we have
\[
\sum_{h<H}
\Biggl\lvert\frac 1R
\sum_{r<R}
\e\bigl(r\bigl(t+h/H\bigr)\bigr)
\Biggr\rvert^2
\leq \frac{H+R-1}{R}.
\]
\end{lemma}
This lemma is an immediate consequence of the analytic form of the large sieve, see \cite[Theorem 3]{M1978}. This form of the theorem, featuring the optimal constant $N-1+\delta^{-1}$, is due to Selberg.
\begin{lemma}[Selberg]
Let $N\geq 1,R\geq 1, M$ be integers,
$\alpha_1,\ldots,\alpha_R\in\dR$ and
$a_{M+1},\ldots,$ $a_{M+N}\in\dC$.
Assume that $\bigl \lVert \alpha_r-\alpha_s\bigr \rVert\geq \delta$ for $r\neq s$.
Then
\[
\sum_{r=1}^R
\Biggl\lvert
\sum_{n=M+1}^{M+N}
a_n\e(n\alpha_r)
\biggr\rvert^2
\leq
\bigl(N-1+\delta^{-1}\bigr)
\sum_{n=M+1}^{M+N}
\bigl\lvert a_n\bigr\rvert^2
.
\]
\end{lemma}
As an important first step in the proof of Theorem~\ref{thm:main}, we show that for the functions in question we have the following uniformity property.
\begin{proposition}\label{prp:uniformity}
Let $g$ be a bounded $\alpha$-multiplicative function.
Assume that the Fourier--Bohr spectrum of $g$ is empty, that is,
\[
\biggl\lvert
\sum_{0\leq n<N}
g(n)\e(-n\beta)
\biggr\rvert
=o(N)
\]
as $N\ra\infty$ for all $\beta\in\dR$. Then
\[
\sup_{\beta\in\dR}
\biggl\lvert
\sum_{0\leq n<N}
g(n)\e(-n\beta)
\biggr\rvert
=o(N).
\]
\end{proposition}
\begin{proof}[Proof of Proposition~\ref{prp:uniformity}]
Without loss of generality we may assume that $\lvert g\rvert\leq 1$, since the full statement follows by scaling.
We first prove the special case 
\[
\lim_{i\rightarrow\infty}
\sup_{\beta\in\dR}
 \frac 1{q_i}
\biggl\lvert
\sum_{0\leq n<q_i}g(n)\e(-n\beta)
\biggr\rvert=0.
\]
We set $h(n)=g(n)\e(-n\beta)$ and
\[
S_i(\beta) =
\frac 1{q_i}
\sum_{0\leq n<q_i}h(n).
\]
For all $i\geq 1$ we have
\begin{align*}
S_{i+1}&=\frac 1{q_{i+1}}
\sum_{0\leq b<a_{i+1}}\sum_{0\leq u<q_i}h(u+bq_i)
+\frac 1{q_{i+1}}
\sum_{0\leq u<q_{i-1}}h(u+a_{i+1}q_i)
\\
&=
\frac {q_i}{q_{i+1}}
\biggl(\sum_{0\leq b<a_{i+1}}h(bq_i)\biggr)
\cdot S_i
+
\frac {q_{i-1}}{q_{i+1}}
h(a_{i+1}q_i)
S_{i-1}.
\end{align*}
Using the recurrence for $q_i$, it follows that
$\lvert S_{i+1}\rvert\leq\max\bigl\{\lvert S_i\rvert,\lvert S_{i-1}\rvert\bigr\}$.
By Lemma~\ref{lem:dini} we obtain the statement.

We pass to the general case.
We consider partial sums of $g(n)\e(n\beta)$ up to $N$.
Assume that $w_i\leq N<w_{i+1}$.
We have
\[
\Biggl\lvert
\sum_{0\leq n<N}g(n)\e(n\beta)
\Biggr\rvert^2
\leq
\Biggl\lvert
\sum_{0\leq n<w_i}g(n)\e(n\beta)
\Biggr\vert^2
+q_\lambda^2
+2Nq_\lambda
.
\]
We apply the inequality of van der Corput (Lemma \ref{lem:vdc}) to obtain
\[
\Biggl\lvert
\sum_{0\leq n<w_i}g(n)\e(n\beta)
\Biggr\rvert^2
\leq
\frac{N+R-1}{R}
\sum_{\lvert r\rvert<R}
\biggl(1-\frac{\lvert r\rvert}{R}\biggr)
\e(r\beta)
\sum_{0\leq n,n+r<w_i}
g(n+r)
\overline{g(n)}
.
\]
We adjust the summation range by omitting the condition $0\leq n+r<w_i$.
This introduces an error term $O(NR)$.
Moreover, $\alpha$-additive functions $f$ satisfy
Lemma \ref{lem:alpha_carry_lemma},
therefore we may replace $g$ by $g_\lambda$ for the price of another error term,
$O(N^2Rq_{\lambda-1}^{-1})$.
Using (\ref{eqn:correlation_partial_sums}) we get
\begin{align*}
\hskip 15pt&\hskip -15pt
\Biggl\lvert
\sum_{0\leq n<w_i}
g(n)\e(n\beta)
\Biggr\rvert^2
\\
&\ll
\frac NR
\sum_{\lvert r\rvert<R}
\biggl(1-\frac{\lvert r\rvert}{R}\biggr)
\e(r\beta)
\Biggl(
\sum_{0\leq n<w_i}
g_\lambda(n+r)\overline{g_\lambda(n)}
+
O\bigl(R+NRq_{\lambda-1}^{-1}\bigr)
\Biggr)
\\
&\ll
NR+N^2\frac R{q_{\lambda-1}}
+
\frac NR
w_i
\sum_{h<q_\lambda}
\bigl \lvert G_\lambda(h)\bigr \rvert^2
\sum_{\lvert r\rvert<R}
\biggl(1-\frac{\lvert r\rvert}{R}\biggr)
\e\biggl(r\biggl(\beta+\frac h{q_\lambda}\biggr)\biggr)
.
\end{align*}
Note that the sum over $r$ is a nonnegative real number.
This follows from the identity
\[
\sum_{\lvert r\rvert<R}(R-\lvert r\rvert)\e(rx)
=
\Biggl\lvert \sum_{0\leq r<R}\e(rx)\Biggr\rvert^2,
\]
which can be proved by 
an elementary combinatorial argument.
We use this equation and collect the error terms to get
\begin{equation}\label{eqn:fourier_coefficient_reduction}
\Biggl\lvert
\frac 1N
\sum_{0\leq n<N}
g(n)\e(n\beta)
\Biggr\rvert^2
\ll
\frac{q_\lambda^2}{N^2}
+\frac{q_\lambda}N
+
\frac RN
+
\frac{R}{q_{\lambda-1}}
+
\sum_{0\leq h<q_\lambda}
\bigl\lvert G_\lambda(h)\bigr\rvert^2
\Biggl\lvert
\frac 1R
\sum_{0\leq r<R}
\e\biggl(r\biggl(\beta+\frac h{q_\lambda}\biggr)\biggr)
\Biggr\rvert^2
.
\end{equation}
Next, using Lemma~\ref{lem:sieve} we get
\begin{equation}\label{eqn:spaced_sum}
\sum_{0\leq h<q_\lambda}
\bigl\lvert G_\lambda(h)\bigr\rvert^2
\Biggl\lvert
\frac 1R
\sum_{0\leq r<R}
\e\biggl(r\biggl(\beta+\frac h{q_\lambda}\biggr)\biggr)
\Biggr\rvert^2
\leq
\sup_{0\leq h<q_\lambda}
\bigl\lvert G_\lambda(h)\bigr\rvert^2
\frac {q_\lambda+R-1}R
.
\end{equation}
Using the special case proved before and choosing $R$ and $\lambda$ appropriately, we obtain the statement.
\end{proof}

In order to establish the existence of the correlation $\gamma_t$ of $g$,
we use the following theorem~\cite[Th\'eor\`eme 4]{CRT1981}.
(Note that we defined $\psi_\lambda(n)=\sum_{0\leq i<\lambda}\varepsilon_i(n)q_i$.)
\begin{lemma}[Coquet--Rhin--Toffin]\label{lem:CRT}
Let $\lambda\geq 1$ and $a<q_\lambda$. The set
$\mathcal E(\lambda,a)=\{n\in\dN:\psi_\lambda(n)=a\}$
possesses an asymptotic density given by
\begin{align*}
\delta&=(q_\lambda+q_{\lambda-1}[0;a_{\lambda+1},\ldots])^{-1}&\mbox{if }a\geq q_{\lambda-1};\\
\delta'&=\delta(1+[0;a_{\lambda+1},\ldots])&\mbox{if }a<q_{\lambda-1}.
\end{align*}
\end{lemma}

\begin{lemma}\label{lem:multiplicative_correlation}
Let $g$ be a bounded $\alpha$-multiplicative function.
Then for every $r\geq 0$ the limit
\[
\lim_{N\to\infty}\frac 1N\sum_{n<N}g(n+r)\overline{g(n)}
\]
exists.
\end{lemma}
We note that the existence of the correlation was established in~\cite{CRT1981} for the special case that $g(n)=\e(y \sigma_\alpha(n))$, where $\e(x)=\e^{2\pi i x}$.
\begin{proof}
Let $\lambda,N\geq 0$ and $r\geq 1$ and set $k=\max\{j:w_j\leq N\}$.
Moreover, let $a=a(N)$ be the number of indices $j<k$ such that $w_{j+1}-w_j=q_\lambda$ and $b=b(N)$ be the number of indices $j<k$ such that $w_{j+1}-w_j=q_{\lambda-1}$.
By Lemma~\ref{lem:CRT} $a(N)/N$ and $b(N)/N$ converge, say to $A$ and $B$ respectively.
Let $\lambda$ be so large that $r/q_{\lambda-1}<\varepsilon$.
Moreover, choose $N_0$ so large that
$\bigl\lvert A-a(N)/N\bigr\rvert<\varepsilon q_\lambda^{-1}$, $\bigl\lvert B-b(N)/N\bigr\rvert<\varepsilon q_{\lambda-1}^{-1}$ and $q_\lambda/N<\varepsilon$ for all $N\geq N_0$.

Then by Lemma \ref{lem:alpha_carry_lemma} we get
\begin{align*}
\sum_{0\leq n<N}g(n+r)\overline{g(n)}
&=
\sum_{0\leq n<N}g_\lambda(n+r)\overline{g_\lambda(n)}
+
O\bigl(Nrq_{\lambda-1}^{-1}\bigr)
\\&=
\sum_{0\leq n<w_k}g_\lambda(n+r)\overline{g_\lambda(n)}
+
O\bigl(q_\lambda+Nrq_{\lambda-1}^{-1}\bigr)
,
\end{align*}
therefore
\begin{align*}
\hskip 10pt&\hskip -10pt
\Biggl\lvert
\frac 1N\sum_{0\leq n<N}g(n+r)\overline{g(n)}
-
A
\sum_{0\leq n<q_\lambda}g_\lambda(n+r)\overline{g_\lambda(n)}
-
B
\sum_{0\leq n<q^{\lambda-1}}g_\lambda(n+r)\overline{g_\lambda(n)}
\Biggr\rvert
\\
&\ll
\Biggl\lvert
\frac 1N\sum_{0\leq n<N}g(n+r)\overline{g(n)}
-
\frac aN
\sum_{0\leq n<q_\lambda}g_\lambda(n+r)\overline{g_\lambda(n)}
-
\frac bN
\sum_{0\leq n<q_{\lambda-1}}g_\lambda(n+r)\overline{g_\lambda(n)}
\Biggr\rvert+2\varepsilon
\\&=
\Biggl\lvert
\frac 1N\sum_{0\leq n<N}g(n+r)\overline{g(n)}
-
\frac 1N
\sum_{0\leq n<w_k}g_\lambda(n+r)\overline{g_\lambda(n)}\Biggr\rvert+2\varepsilon
\\&\ll \frac {q_\lambda}{N}+\frac r{q_{\lambda-1}} +2\varepsilon.
\end{align*}
By the triangle inequality it follows that the values
$\frac 1N\sum_{n<N}g(n+r)\overline{g(n)}$
form a Cauchy sequence and therefore a convergent sequence,
which proves the existence of the correlation of $g$.
\end{proof}
\section{Proof of the theorem}\label{sec:thmproof}
Now we are prepared to prove Theorem~\ref{thm:main}.
If $g$ is pseudorandom, then by Lemma~\ref{lem:pseudorandom_spectrum} its spectrum is empty. We are therefore concerned with the converse.
Let $\ell\geq 0$. We denote by $a$ the number of $i<\ell$ such that
$w_{i+1}-w_i=q_\lambda$ and by $b$ the number of $i<\ell$ such that
$w_{i+1}-w_i=q_{\lambda-1}$.

Choose $\varepsilon_r$ such that $\lvert \varepsilon_r\rvert=1$ and
\[\varepsilon_r
  \sum_{0\leq h<q_\lambda}{
    \bigl\lvert G_\lambda(h)\bigr\rvert^2\e\bigl(hrq_\lambda^{-1}\bigr)}
\]
is a nonnegative real number.
Similarly choose $\varepsilon'_r$ for $\lambda-1$.
We have
\begin{align*}
\hskip 30pt&\hskip -30pt
  \frac 1R\sum_{0\leq r<R}
  \Biggl\lvert
    \frac 1{w_\ell}\sum_{0\leq n<w_\ell}
    g_\lambda(n+r)\overline{g_\lambda(n)}
  \Biggr\rvert
\\&=
\Biggl\lvert
  \frac {aq_\lambda}{w_\ell}
  \frac 1R\sum_{0\leq r<R}
  \varepsilon_r
  \sum_{0\leq h<q_\lambda}{
    \bigl\lvert G_\lambda(h)\bigr\rvert^2\e\biggl(\frac{hr}{q_\lambda} \biggr)}
\\&+
  \frac {bq_{\lambda-1}}{w_\ell}
  \frac 1R\sum_{0\leq r<R}
  \varepsilon'_r
  \sum_{0\leq h<q_{\lambda-1}}{
    \Bigl\lvert G_{\lambda-1}(h)\Bigr\rvert^2\e\biggl(\frac{hr}{q_{\lambda-1}}\biggr)
  }
\Biggr\rvert
+
O\biggl(\frac {ar}{w_\ell}+\frac{br}{w_\ell}\biggr)
\\&=
\frac 1R
\Biggl\lvert
  \frac {aq_{\lambda}}{w_\ell}
  \sum_{0\leq h<q_\lambda}{
    \bigl\lvert G_\lambda(h)\bigr\rvert^2
    \sum_{0\leq r<R}{
      \varepsilon_r
      e\biggl(\frac{hr}{q_\lambda}\biggr)
    }
  }
\\&+
  \frac {bq_{\lambda-1}}{w_\ell}
  \sum_{0\leq h<q_{\lambda-1}}{
    \bigl\lvert G_{\lambda-1}(h)\bigr\rvert^2
    \sum_{0\leq r<R}{
      \varepsilon'_r
      e\biggl(\frac{hr}{q_{\lambda-1}}\biggr)
    }
  }
\Biggr\rvert
+
O\biggl(\frac r{q_{\lambda-1}}\biggr)
\\&
\leq
\frac 1R
\Biggl\lvert
  \sum_{0\leq h<q_\lambda}{
    \Bigl\lvert G_\lambda(h)\bigr\rvert^2
    \sum_{0\leq r<R}{
      \varepsilon_r
      e\biggl(\frac{hr}{q_\lambda}\biggr)
    }
  }
\Biggr\rvert
\\&+
\frac 1R
\Biggl\lvert
  \sum_{0\leq h<q_{\lambda-1}}{
    \bigl\lvert G_{\lambda-1}(h)\bigr\rvert^2
    \sum_{0\leq r<R}{
      \varepsilon'_r
      e\biggl(\frac{hr}{q_{\lambda-1} }\biggr)
    }
  }
\Biggr\rvert
+
O\biggl(\frac r{q_{\lambda-1} }\biggr)
.
\end{align*}
By Cauchy-Schwarz we obtain
\begin{align*}
\hskip 15pt&\hskip -15pt
\frac 1{R^2}
\Biggl\lvert
  \sum_{0\leq h<q_\lambda}{
    \bigl\lvert G_\lambda(h)\bigr\rvert^2
    \sum_{0\leq r<R}{
      \varepsilon_r
      e\biggl(\frac{hr}{q_\lambda}\biggr)
    }
  }
\Biggr\rvert^2
\\
&\leq
\frac 1{R^2}
\sum_{0\leq h<q_\lambda}\bigl\lvert G_\lambda(h)\bigr\rvert^4
\sum_{0\leq h<q_\lambda}
\Biggl\lvert
  \sum_{0\leq r<R}{
    \varepsilon_{r}
    \e\biggl(\frac{hr}{q_\lambda}\biggr)
  }
\Biggr\rvert^2
\\
&\leq
\frac 1{R^2}
\sum_{0\leq h<q_\lambda}\bigl \lvert G_\lambda(h)\bigr \rvert^4
\sum_{0\leq h<q_\lambda}
\sum_{0\leq r_1,r_2<R}
\varepsilon_{r_1}\overline{\varepsilon_{r_2}}
\e\biggl(h\frac{r_1-r_2}{q_\lambda}\biggr)
\\
&=
\frac {q_\lambda}{R^2}
\sum_{0\leq h<q_\lambda}\bigl \rvert G_\lambda(h)\bigr \rvert^4
\sum_{0\leq r_1,r_2<R}
\varepsilon_{r_1}\overline{\varepsilon_{r_2}}
\delta_{r_1,r_2}
\\
&=
\frac {q_\lambda} R
\sum_{0\leq h<q_\lambda}\bigl \lvert G_\lambda(h)\bigr \rvert^4,
\end{align*}
similarly for $\lambda-1$.
Using Lemma \ref{lem:alpha_carry_lemma}, we get
\begin{align*}
\hskip 15pt&\hskip -15pt
\frac 1R\sum_{0\leq r<R}
\bigl\lvert \gamma_r \bigr\rvert
=
\lim_{\ell\ra\infty}
\frac 1R\sum_{0\leq r<R}
\Biggl\lvert
\frac 1{w_\ell}\sum_{0\leq n<w_\ell}g(n+r)\overline{g(n)}
\Biggr\rvert
\\
&=
\lim_{k\ra\infty}
\frac 1R\sum_{0\leq r<R}
\Biggl\lvert
\frac 1{w_\ell}\sum_{0\leq n<w_\ell}g_\lambda(n+r)\overline{g_\lambda(n)}
\Biggr\rvert
+O\biggl(\frac R{q_{\lambda-1}}\biggr)
\\
&\leq
\vastl[
\Biggl(\sum_{0\leq h<q_{\lambda-1}}\bigl \lvert G_{\lambda-1}(h)\bigr \rvert^4\Biggr)^{1/2}
+
\Biggl(\sum_{0\leq h<q_\lambda}\bigl \lvert G_\lambda(h)\bigr \rvert^4\Biggr)^{1/2}
\vastr]
\biggl(
\frac {q_\lambda}R
\biggr)^{1/2}
+
O\biggl(\frac R{q_{\lambda-1}}\biggr).
\end{align*}
Using the hypothesis of the theorem and Proposition~\ref{prp:uniformity}, we get
$\sup_h \bigl\lvert G_\lambda(h)\bigr\rvert=o(1)$ as $\lambda\rightarrow\infty$.
By Parseval's identity this implies
\[
\sum_{0\leq h< q_\lambda} |G_\lambda(h)|^4 =o(1).
\]
By a straightforward argument
we conclude that
\[
\frac 1R\sum_{0\leq r<R}
\bigl\lvert
\gamma_r
\bigr\rvert
=o(1)
\]
as $R\ra\infty$. 
Since $g$ is bounded, an application of Lemma~\ref{lem:L1} completes the proof of Theorem~\ref{thm:main}.


\end{document}